\documentclass[reqno, 11pt]{amsart}
\usepackage{hyperref}
\usepackage{xcolor}
\usepackage{enumitem}

\usepackage[T1]{fontenc}
\usepackage[a4paper,hmargin={2.7cm,2.7cm},vmargin={3.3cm,3.3cm}]{geometry}
\usepackage{amsmath}
\usepackage{amssymb}   
\usepackage{amsthm}    
\usepackage{thmtools}  
\usepackage{mathtools} 
\usepackage{mathrsfs}  
\usepackage{commath}
\usepackage{bbm}

\RequirePackage[backend    = biber,
                sortcites  = true,
                giveninits = true,
                doi        = false,
                isbn       = false,
                url        = false,
                maxnames   = 50,
                citestyle  = numeric-comp]{biblatex}
\DeclareNameAlias{sortname}{family-given}
\DeclareNameAlias{default}{family-given}
\usepackage{varioref}
\usepackage{hyperref}
\usepackage[nameinlink, capitalize, noabbrev]{cleveref}

\theoremstyle{plain}
\declaretheorem[numberwithin = section]{theorem}

\declaretheorem[sibling = theorem]{lemma}
\declaretheorem[sibling = theorem]{proposition}
\theoremstyle{definition}

\declaretheorem[sibling = theorem]{remark}
\declaretheorem[sibling = theorem]{example}

\def \C{ \mathbb{C} }

\def \M{ \mathcal{M} }
\def \H{ \mathcal{H} }

\def \eps{ \varepsilon }
\def \span{ \operatorname{span} }
\def \widecheck{ \widetilde }
\def \dif#1{ \,\mathrm{d}#1 }

\def \Gbracket#1#2{ \,\!_{\alpha}\!\langle #1,#2 \rangle }
\def \Hbracket#1#2{ \langle #1,#2 \rangle_{\beta} }

\title[Associativity of operator convolutions]{Associativity of operator convolutions for group actions on von Neumann algebras}

\author{Hannes Wendt}
\address{Department of Mathematics, University of Oslo, Moltke Moes vei 35, 0851 Oslo.}
\email{hhwendt@math.uio.no}

\begin{document}

\begin{abstract}
    Given a pair of commuting, ergodic, trace-preserving and trace-integrable actions of locally compact unimodular groups on a common semifinite von Neumann algebra we prove a form of associativity for convolutions of triples of operators. We view this as a basis for a general version of quantum harmonic analysis.
\end{abstract}

\maketitle
\thispagestyle{empty}

\section{Introduction}

Quantum harmonic analysis (QHA) is a framework introduced by Werner \cite{We84} to demonstrate parallels between classical harmonic analysis and quantum mechanics.
Starting from a projective representation of the phase space $G = \mathbb R^{2d}$ on $\H = L^2(\mathbb R^d)$ (or the Schrödinger representation of the Heisenberg group) a shift-action is defined for bounded operators on $\H$, which naturally leads to convolutions between functions and operators. These allowed Werner an intuitive transfer of classical results into quantum analogues.
More recently, QHA has also found fruitful connections to time-frequency analysis and wavelet theory \cite{LuSk18, LuSk19, LuSk20, LuSk21, BeBeLuSk22}.

To justify the systematic use of the convolution formalism, the compatibility between the function-operator convolutions and operator-operator convolutions is central. In particular, Werner shows for triples of trace-class operators $S,T,R$ on $\H$ that \begin{equation}
    \label{eq:werners-operator-associativity}
    (S * T) * R = S * (T * R).
\end{equation}
This associativity is especially important for his version of ``Wiener's approximation theorem'' and its applications, as e.g. for the physical distinguishability of quantum states \cite{KiLaSchuWer12}. 

Starting with the development of an affine version of QHA \cite{BeBeLuSk22}, there has been interest in generalizations based on irreducible, square-integrable representations of locally compact groups \cite{FuGa25, Ha23}. While \eqref{eq:werners-operator-associativity} has been shown in \cite{We84, LuSk18, FuGa25} where the operator-shifts are actions of abelian groups, no corresponding condition has been established in situations with more general groups. 
This work aims to explain this difference, by proposing a broadened setup, in which we find necessary and sufficient conditions for a form of \eqref{eq:werners-operator-associativity} to hold.

Our extended setup rests on a recent joint work with U. Enstad \cite{EnWe25}: There we showed for appropriate actions $\alpha\colon G\curvearrowright \M$ of locally compact groups on von Neumann algebras with semifinite trace $\tau$ that the scalar-valued functions $\Gbracket{x}{y}(s) = \tau(x \alpha_s(y^*)), s \in G$ are integrable for all trace-class elements $x$ and admissible operators $y$. As this generalises a result of Duflo and Moore \cite{DuMo1976}, it yields Moyal's identity for the operator-shift of QHA \cite[Proposition 4.3.2]{Gr01} and therefore strongly motivates us to view $\Gbracket{\,\cdot\,}{\,\cdot\,}$ as an analogue of Werner's operator-operator convolution. In the present work, we complement the setup of \cite{EnWe25} by defining compatible function-operator convolutions $f * x$.

To formulate the analogue of \eqref{eq:werners-operator-associativity}, we consider a two-sided setup with the above language: We assume a pair of commuting, $\tau$-preserving actions \[
    \alpha\colon G \curvearrowright \M, \quad 
    \beta\colon H \curvearrowright \M,
\] of locally compact groups $G$ and $H$ on a common von Neumann algebra $\M$. These define commuting left- and right- function-operator convolutions, which may be genuinely distinct already for $\M = L^\infty(G)$ over non-abelian groups $G$. 
We then show that it is impossible for both $\alpha$ and $\beta$ to satisfy the conditions of the Duflo--Moore theorem of \cite{EnWe25} when one of the groups $G$ or $H$ is non-unimodular.

On the contrary, when commuting, trace-preserving actions $\alpha$ and $\beta$ of unimodular groups do satisfy the conditions of the Duflo--Moore theorem, we may define integrable functions $\Gbracket{\,\cdot\,}{\,\cdot\,}$ and $\Hbracket{\,\cdot\,}{\,\cdot\,}$ for all pairs of trace-class operators, and we show in our main result \cref{thm:associativity} the compatibility \[
    \Gbracket{x}{y} * z = x * \Hbracket{y}{z}.
\] 
This recovers the associativity \eqref{eq:werners-operator-associativity} for convolutions of operators when applied the setting of QHA, and simultaneously the associativity of function-convolutions when considered for shift-actions on commutative $\M$.

From the bimodule structure of the function-convolutions on bounded, trace-class operators of $\M$, we deduce Morita equivalence of ideals in the reduced $C^*$-algebras over the acting groups. As a further application, and in the spirit of QHA, we show a version of Wiener's approximation theorem, as in \cite[Proposition 3.5]{We84} and \cite{KiLaSchuWer12}. As our setup does not consider Fourier transforms of operators, this amounts to equivalent formulations of regularity of operators. We apply this to prove a Tauberian theorem for bounded operators, following \cite[Sect. 5]{LuSk21}.

\subsection*{Outline}

The preliminary \cref{sect:preliminaries} recalls necessary basics on von Neumann algebras and explains the generalised Duflo--Moore theorem and its assumptions. Actually, in the \nameref{sect:appendix} we demonstrate that the assumptions can not be weakened.

In \cref{sect:convolutions} we define the function-operator convolutions first for a trace-preserving group action $\alpha$, and analogous as in \cite[Sect. III]{We84}, \cite[Sect. 4]{LuSk18}, assuming the conditions of the Duflo--Moore theorem, we extend the definition to its most generality via the duality by $\Gbracket{\,\cdot\,}{\,\cdot\,}$.

Both the impossibility of the triple convolution for non-unimodular groups, as well as our main result \cref{thm:duflo-moore}, we show in \cref{sect:associativity} for pairs of commuting, trace-preserving actions on the common semifinite von Neumann algebra.

Finally, \cref{sect:bimodule} contains the application of the associativity for Morita equivalence, while \cref{sect:regularity} and \cref{sect:tauberian} contain Wiener's approximation theorem and the Tauberian-type theorem.

\subsection*{Acknowledgements}

The author is indebted to Ulrik Enstad and Makoto Yamashita for many valuable discussions.

\section{Preliminaries}\label{sect:preliminaries}

In the following, $G$ always denotes a locally compact, second countable group, furnished with some left-invariant Haar measure $m$. We denote by $\Delta$ the modular function associated to $m$, the unique continuous homomorphism from $G$ into $\mathbb R_{ > 0}$ satisfying $m(B s) = \Delta(s) m(B)$ for all $s \in G$ and measurable subsets $B \subseteq G$. Whenever the Haar measure is bi-invariant, that is $\Delta \equiv 1$, one calls $G$ unimodular.

We also consider semifinite von Neumann algebras $\M$ with separable predual. By weights on a von Neumann algebra we understand affine maps from the positive cone $\M_+$ to $[0,\infty]$. We fix a faithful, normal, semifinite weight $\tau$ on $\M$ that satisfies the trace property $\tau(x^* x) = \tau(x x^*)$ for all $x \in \M$. This trace linearly extends to $\mathfrak m_\tau$, the ultraweakly dense ideal in $\M$ comprised of all elements $x \in \M$ satisfying $\tau(|x|) < \infty$.

By $L^p(\M)$ we denote the completion of $\mathfrak m_\tau$ with respect to the norm $\|x\|_p = \tau(|x|^p)^{1/p}$ for $x \in \mathfrak m_\tau$, $1 \leq p < \infty$. We formally put $L^\infty(\M) = \M$. The GNS-construction for $\M$ with respect to $\tau$ coincides with the Hilbert space $L^2(\M)$, which we will use when concretely representing $\M \subseteq B(L^2(\M))$ as bounded multiplication operators. We refer to \cite{Ne74} for an exposition of noncommutative integration.

The space of \emph{affiliated} operators $\widetilde \M$ is the collection of unbounded, densely defined, closed linear operators on $L^2(\M)$ that commute with the commutant $\M'$ of $\M$. 
The self-adjoint, positive $A \in \widetilde \M_+$ yield increasing sequences $A_\eps = A(1+\eps A)^{-1} \in \M_+$, and via $\tau_A(x) = \sup_{\eps > 0} \tau( (A_\eps)^{1/2} x (A_\eps)^{1/2} ), x \in \M_+$ define semifinite, normal weights $\tau_A$.

\begin{theorem}[Radon--Nikodym theorem \cite{PeTa73}]
    Let $\M$ be a von Neumann algebra with faithful, normal, semifinite trace $\tau$. For every normal, semifinite weight $\psi$ on $\M$ there exists a unique self-adjoint, positive operator $A$ affiliated to $\M$ such that $\psi = \tau_A$.
\end{theorem}

We will consider group actions on von Neumann algebras:
By $\alpha\colon G \curvearrowright \M$ we denote group homomorphisms from $G$ into the group of automorphisms of $\M$. They are always assumed to be ultraweakly continuous, that is, for all $y \in \M$ and all $x \in L^1(\M)$ we assume that the bounded maps $s \mapsto \tau(x \alpha_s(y))$ are continuous on $G$. We say that \begin{enumerate}
    \item $\alpha$ is \emph{ergodic}, if the fixed-point algebra of $\alpha$ inside $\M$ equals $\mathbb C 1$.
    \item $\alpha$ is \emph{$\tau$-preserving}, if $\tau(\alpha_s(x)) = \tau(x)$ for all $x \in \M_+$ and $s \in G$.
    \item $\alpha$ is \emph{$\tau$-integrable}, if there exists a pair of nonzero elements $x,y \in \M_+$ such that the non-negative function $s \mapsto \tau(x^{1/2} \alpha_s(y) x^{1/2})$ is integrable with respect to Haar measure.
\end{enumerate}
The third notion was introduced in \cite{EnWe25}. We cite the Duflo--Moore theorem for positive elements:

\begin{theorem}[Duflo--Moore theorem \cite{EnWe25}]\label{thm:duflo-moore}
    Let $\alpha\colon G \curvearrowright \M$ be an action of a locally compact group on a von Neumann algebra with semifinite trace $\tau$. If $\alpha$ is $\tau$-preserving, ergodic and $\tau$-integrable, then there exists a unique, densely defined, self-adjoint, positive, invertible operator $D$ affiliated to $\M$ such that \begin{equation}\label{eqn:duflo-moore}
        \int_G \tau(x \alpha_s(y)) \dif m(s) = \tau(x) \tau_{D^{-1}}(y), \quad x \in L^1(\M)_+, y \in \M_+.
    \end{equation} Furthermore, $G$ is unimodular if and only if $D$ is a scalar multiple of the identity.
\end{theorem}

Having fixed an action $\alpha$ as in the theorem, all elements $y \in \M_+$ with $\tau_{D^{-1}}(y) < \infty$ are called the \emph{positive admissible elements of $\M$}. An \emph{admissible element of $\M$} is a linear combinations of positive admissible elements of $\M$. The Banach-space completion of the admissible elements of $\M$ with respect to the norm $\|y\|_{1, D^{-1}} = \lim_{\epsilon \to 0} \|(D^{-1})^{1/2}_\epsilon y (D^{-1})^{1/2}_\epsilon\|_1$ is denoted $L^1(\M, D^{-1})$, called the set of \emph{admissible elements}, and there is a sesquilinear map \[
    \Gbracket{\,\cdot\,}{\,\cdot\,}\colon L^1(\M) \times L^1(\M, D^{-1}) \to L^1(G, m),
\] such that $\Gbracket{x}{y}$ and $s \mapsto \tau(x \alpha_s(y^*))$ coincide as elements of $L^1(G, m)$ for all $x \in L^1(\M)$ and all admissible $y \in \M$. In fact, for $x \in L^1(\M)$ and admissible $y \in \M$ we have that \[
    \|\Gbracket{x}{y}\|_1 \leq \|x\|_1 \|y\|_{1, D^{-1}}, \quad \int_G \Gbracket{x}{y} \dif m = \tau(x) \overline{\tau_{D^{-1}}(y)}.
\]

\section{Main results}
\subsection{Convolutions}\label{sect:convolutions}

Let $\alpha \colon G \curvearrowright \M$ be a $\tau$-preserving action on $\M$. We will first define function-operator convolutions \[
    f * y \in L^p(\M), \quad y \in L^p(\M), f \in L^1(G, m),
\] where $1 \leq p \leq \infty$. For $\tau$-preserving actions $\alpha$ of unimodular groups $G$ which are additionally ergodic and $\tau$-integrable, we will afterwards extend the definition to \[
    f * y \in L^r(\M), \quad y \in L^p(\M), f \in L^q(G, m),
\] where $1 \leq p,q \leq r \leq \infty$ are such that $1/p + 1/q = 1 + 1/r$.

\begin{remark}
    Given $f \in L^1(G, m)$ and $y \in L^1(\M)$ one may appeal to \cite[Theorem A3.3]{Fo95} to define $f * y$ as the unique element of the Banach space $L^1(\M)$ that satisfies \[
        \phi(f * y) = \int_G f(s) \phi(\alpha_s(y)) \dif m(s)
    \] for all continuous, linear functionals $\phi$ on $L^1(\M)$. This approach is taken in \cite[Sect. 2.3]{LuSk18}. In their concrete situation, with $\M = B(\H), \tau = \operatorname{Tr}$, it happens that $L^1(\M) \subseteq \M$. Given a general von Neumann algebra $\M$, we want to first define convolutions with $\mathfrak m_\tau = L^1(\M) \cap \M$.
\end{remark}

Recall that we represent $\M$ as multiplication operators on $L^2(\M)$. Let $y \in \M$. By the ultraweak continuity of $\alpha$, for all $\xi,\eta \in L^2(\M)$ the bounded function $s \mapsto \tau(\eta^* \alpha_s(y) \xi)$ on $G$ is continuous. So, given $f \in L^1(G, m)$, we define \[
    \langle (f * y) \xi, \eta \rangle_{L^2(\M)} = \int_G f(s) \tau(\eta^* \alpha_s(y) \xi) \dif m(s).
\] This defines a bounded, linear operator $f * y$ on $L^2(\M)$ with $\| f * y \| \leq \|f\|_1 \|y\|_\infty$. Moreover, for all bounded operators $A$ on $L^2(\M)$ that commute with $\M$, we find that \[
    \langle (f * y) A \xi, \eta \rangle 
    = \int_G f(s) \tau(\eta^* \alpha_s(y) A \xi) \dif m(s)
    = \int_G f(s) \tau(\eta^* A \alpha_s(y) \xi) \dif m(s)
    = \langle (f * y) \xi, A^* \eta \rangle.
\] This implies $f * y \in (\M')' = \M$. We note that $(f * y)^* = \overline{f} * y^*$ for $f \in L^1(G)$ and $y \in \M$.
Furthermore, for all $x \in \mathfrak m_\tau$ we find $x = \xi \eta^*$ for some pair $\xi,\eta \in \M \cap L^2(\M)$, so that by the trace-property it holds that $\tau(\eta^* \alpha_s(y) \xi) = \tau(x \alpha_s(y))$ for all $s \in G$. Consequently, we have
\begin{equation}\label{eq:weak-function-operator-definition}
    \tau(x (f*y)) = \int_G f(s) \tau(x \alpha_s(y)) \dif m(s).
\end{equation}

\begin{lemma}
    Let $f \in L^1(G, m)$. If $y \in \mathfrak m_\tau$, then $f * y \in \mathfrak m_\tau$, and for all $1 \leq p \leq \infty$ it holds \[
        \|f * y\|_p \leq \|f\|_1 \|y\|_p.
    \]
    In particular, the linear map $y \mapsto f * y$ continuously extends to a map on $L^p(\M)$.
\end{lemma}
\begin{proof}
    Combining \eqref{eq:weak-function-operator-definition} with $|\tau(x \alpha_s(y))| \leq \|x\|_\infty \|y\|_1$ for all $s \in G$ and $x \in \mathfrak m_\tau$, we find that \[
        |\tau(x (f*y))| = \left| \int_G f(g) \tau(x \alpha_s(y)) \dif m(s) \right| \leq \int_G |f(s)| |\tau(x \alpha_s(y))| \dif m(s) \leq \|f\|_1 \|x\|_\infty \|y\|_1.
    \]
    Consider the polar decomposition $f * y = u |f * y|$ with partial isometry $u \in \M$ and the positive element $|f * y| \in \M_+$. Using that $\tau$ is semifinite, we may pick a net of projections $p_i \in \M_+$ with $\tau(p_i) < \infty$ and least upper bound $1 \in \M_+$. Then consider $x_i = p_i u^* \in \M$. It holds that $\|x_i\|_\infty \leq 1$, so that from the above bound it follows \[
        0 \leq \tau(p_i |f * y|)
        = \tau(x_i (f * y))
        \leq \|f\|_1 \|y\|_1.
    \] Taking the supremum over the net $(p_i)_i$, this implies $f * y \in \mathfrak m_\tau$ with $\tau(|f * y|) \leq \|f\|_1 \|y\|_1$.
    
    The claimed norm inequality only remains to be shown for $1 < p < \infty$. Assuming $f * y \neq 0$, we may consider $x = \|f * y\|_p^{1-p} |f * y|^{p-1} u^* \in \mathfrak m_\tau$. Then $\|x\|_q \leq 1$, where $1 < q < \infty$ is conjugate such that $1/p + 1/q = 1$. Then by Hölder's inequality we have $|\tau(x \alpha_s(y))| \leq \|y\|_p$ for all $s \in G$. We again use the triangle inequality on \eqref{eq:weak-function-operator-definition} to deduce \[
        \|f * y\|_p = \|f * y\|_p^{1-p} \|f * y\|_p^p 
        = \tau(x (f*y)) 
        \leq \int_G |f(s)| |\tau(x \alpha_s(y))| \dif m(s)
        \leq \|f\|_1 \|y\|_p.
    \]
\end{proof}

Denote $(\lambda_t f)(s) = f(t^{-1} s)$ for $s,t \in G$ and $f \in L^1(G, m)$. One deduces from the $\tau$-invariance of $\alpha$ and left-invariance of the Haar measure that \[
    \alpha_t(f * y) = (\lambda_t f) * y, \quad t \in G, y \in \mathfrak m_\tau.
\] It follows that the function-operator convolutions makes $L^p(\M)$ a left $L^1(G, m)$-module, as \[
    g * (f * y) = (g * f) * y, \quad f,g \in L^1(G, m), y \in L^p(\M).
\]
We will later use the language of \cite{EnWe25} and denote the sesquilinear map \[
    \Gbracket{\,\cdot\,}{\,\cdot\,}\colon \mathfrak m_\tau \times \mathfrak m_\tau \to C_b(G), \quad \Gbracket{x}{y}(s) = \tau(x \alpha_s(y^*)), s \in G.
\]
By using the $\tau$-invariance of $\alpha$ and applying \eqref{eq:weak-function-operator-definition}, it follows that \begin{equation}\label{eq:convolution-bracket-compatibility}
    f * \Gbracket{x}{y} = \Gbracket{f * x}{y}, \quad f \in L^1(G), x,y \in \mathfrak m_\tau.
\end{equation}

For the rest of this subsection assume that $G$ is a unimodular group, and that $\alpha\colon G \curvearrowright \M$ is a $\tau$-preserving, ergodic and $\tau$-integrable action. We let $d > 0$ be the positive constant such that $D = d 1$ is the bounded Duflo--Moore operator for the action $\alpha$, cf. \cref{thm:duflo-moore}.

Let $y \in L^p(\M)$ and $1/p + 1/q = 1 + 1/r$ for $1 \leq p,q \leq r \leq \infty$. Whenever $q = 1$, the function-operator convolutions have already been defined above. Otherwise, let $r', q' \geq 1$ be such that $1/r + 1/r' = 1$ and $1/q + 1/q' = 1$. One checks $1/r' + 1/p = 1 + 1/q'$, so that by a result of \cite{EnWe25}
$\Gbracket{\,\cdot\,}{\,\cdot\,}\colon \mathfrak m_\tau \times \mathfrak m_\tau \to L^1(G) \cap L^\infty(G)$
extends to a sesquilinear map $L^{r'}(\M) \times L^p(\M) \to L^{q'}(G)$ and satisfies the bound \begin{equation}\label{eq:bracket-continuous}
    d^{1/q'} \, \|\Gbracket{x}{y}\|_{q'} \leq \|x\|_{r'} \|y\|_{p}.
\end{equation} Since $r, q > 1$, the bounded, linear map $
    T_y \colon L^{r'}(\M) \to L^{q'}(G), T_y(x) = \Gbracket{x}{y}
$
has a Banach-space adjoint map $
    T_y^*\colon L^q(G) \to L^r(\M). 
$ 
We define $f * y = T_y^*(f), f \in L^q(G)$. Then \[
    d^{1/q'} \, \|f * y\|_r \leq \|f\|_q \|y\|_p.
\]
As before, we have an equality of the form \eqref{eq:weak-function-operator-definition}: For all $x,y \in \mathfrak m_\tau$ it holds \[
    \tau(x (f*y)) = \langle T_y^*(f), x^* \rangle 
    = \langle f, T_y(x^*) \rangle 
    = \int_G f(s) \overline{T_y(x^*)(s)} \dif m(s) 
    = \int_G f(s) \tau(x \alpha_s(y)) \dif m(s).
\] 
As can be seen in \cite[Theorem 4.7]{LuSk18}, \cite[Proposition 4.14]{Ha23} when $\M = B(\H)$ and $\alpha_s(x) = \pi_s^* x \pi_s$ for some irreducible, square-integrable unitary representations $\pi$ of $G$, the above is equivalent to definitions of function-operator convolutions of quantum harmonic analysis.

We record for later, that the integrated action $\alpha$, viewed as representation of $G$ on $L^2(\M)$, is weakly contained in the left-regular representation $\lambda$ of $G$ on $L^2(G)$:

\begin{lemma}\label{lemma:weak-containment}
    Let $G$ be a unimodular group, and let $\alpha \colon G \curvearrowright \M$ be a $\tau$-preserving, ergodic and $\tau$-integrable action. Then it holds $\|f * x\|_2 \leq \|\lambda(f)\|_{B(L^2(G))} \|x\|_2$ for all $f \in L^1(G), x \in \mathfrak m_\tau$.
\end{lemma}
\begin{proof}
    Let $x \in \mathfrak m_\tau$. The continuous, bounded function $\phi = \Gbracket{x}{x} \in C_b(G)$ is integrable by the Duflo--Moore theorem. Furthermore, $\phi$ is positive definite: For all $n \in \mathbb N$ and all choices of $c_1, \dots, c_n \in \mathbb \C$, $s_1, \dots, s_n \in G$ it holds 
    \[
        \sum_{i,j = 1}^n c_i \overline{c_j} \phi(s_i^{-1} s_j)
        = \sum_{i,j = 1}^n c_i \overline{c_j} \tau(\alpha_{s_i}(x) \alpha_{s_j}(x)^*)
        = \tau(\Big| \sum_{i = 1}^n c_i\alpha_{s_i}(x) \Big|^2) \geq 0.
    \] By a theorem of Godement \cite{Go48} (see also \cite[Theorem 13.8.6]{Di64}) there exists $g \in L^2(G)$ such that for all $s \in G$ it holds $\phi(s) = \int_G g(t) \overline{g(s^{-1} t)} \dif t$. Now for $f \in L^1(G)$ with \eqref{eq:weak-function-operator-definition} and \eqref{eq:convolution-bracket-compatibility} we compute: \[
        \|f * x\|_2^2 
        = \tau( (f * x) (f * x)^* )
        = \int_G \overline{f(s)} \Gbracket{f * x}{x}(s) \dif s
        = \int_G \overline{f(s)} (f * \phi)(s) \dif s
        = \| f * g \|_2^2.
    \]
    Thus with $\|x\|_2^2 = \phi(e) = \|g\|_2^2$ it follows $\|f * x\|_2 = \|\lambda(f) g\|_2 \leq \|\lambda(f)\| \|g\|_2 = \|\lambda(f)\| \|x\|_2$.
\end{proof}

\subsection{Commuting actions}\label{sect:associativity}

In this subsection, we consider a pair of $\tau$-preserving actions of locally compact, second countable groups $G$ and $H$ on a common von Neumann algebra $\M$ \[
    G \stackrel\alpha\curvearrowright \M, \quad 
    H \stackrel\beta\curvearrowright \M,
\] which \emph{commute}, in the sense that for all $s \in G, t \in H$ and $x \in \M$ it holds that \[
    \alpha_s(\beta_t(x)) = \beta_t(\alpha_s(x)).
\] 

Denote by $m_G$ and $m_H$ some left-invariant Haar measures of $G$ and $H$ respectively. We use the construction of \cref{sect:convolutions} to define left- and right-convolutions on $L^p(\M)$. As we choose to write both $\alpha,\beta$ as left-actions, for the right-convolutions we denote $f \in L^1(H, \widecheck m_H)$ whenever we want to consider an (equivalence class of a) measurable function $f$ on $H$ such that $\int_H |f(t^{-1})| \dif m_H(t) < \infty$. We recall that this is equivalent to $\Delta_H^{-1} f \in L^1(H, m_H)$.

Both $L^1(G, m_G)$ and $L^1(H, \widecheck m_H)$ are $*$-algebras under convolution, with involutions defined for $f_1 \in L^1(G, m_G), f_2 \in L^1(H, \widecheck m_H)$ by $f_1^*(s) = \Delta^{-1}(s) \overline{f_1(s^{-1})}$ and $f_2^*(t^{-1}) = \Delta^{-1}(t) \overline{f_2(t)}$.

Let now $y \in L^p(\M)$ with $1 \leq p \leq \infty$. Then, we have weakly defined elements of $L^p(\M)$: \[
    f_1 * y = \int_G f_1(s) \alpha_s(y) \dif m_G(s),
    \quad
    y * f_2 = \int_H f_2(t^{-1}) \beta_t(y) \dif m_H(t),
\]
for all $f_1 \in L^1(G, m_G)$ and $f_2 \in L^1(H, \widecheck m_H)$.
We emphasize that as $\alpha$ and $\beta$ are assumed to commute, it holds that \begin{equation*}
    (f_1 * y) * f_2 = f_1 * (y * f_2).
\end{equation*}

We follow the notation of \cite{EnWe25} and consider sesquilinear maps $\Gbracket{\,\cdot\,}{\,\cdot\,}$ and $\Hbracket{\,\cdot\,}{\,\cdot\,}$, defined for all $x,z \in L^p(\M), y \in L^q(\M)$ with $1 \leq p,q \leq \infty$, $1/p+1/q = 1$ as continuous, bounded functions
\begin{equation}\label{eq:brackets-definition}
    \Gbracket{x}{y}(s) = \tau(x \alpha_s(y^*)), \; s \in G, 
    \quad 
    \Hbracket{y}{z}(t) = \tau(\beta_{t^{-1}}(y^*) z), \; t \in H.
\end{equation}
We record the following properties:
For all $f_0, f_1 \in L^1(G, m_G)$, $f_2, f_3 \in L^1(H, \widecheck m_H)$ and all elements $x,z \in L^p(\M), y \in L^q(\M)$ it holds that \[
    f_0 * (f_1 * x) = (f_0 * f_1) * x, \quad
    (z * f_2) * f_3 = z * (f_2 * f_3),
\] \[
    f_1 * \Gbracket{x}{y} = \Gbracket{f_1 * x}{y}, \quad
    \Hbracket{y}{z * f_2} = \Hbracket{y}{z} * f_2.
\] \[
    \Gbracket{x * f_2}{y} = \Gbracket{x}{y * f_2^*}, \quad
    \Hbracket{y}{f_1 * z} = \Hbracket{f_1^* * y}{z}.
\]

\begin{example}\label{ex:qha-example}
    Consider any von Neumann algebra $\M$ with semifinite trace $\tau$ that is acted on by an abelian group $G$ via some $\tau$-preserving action $\alpha$. One may set $H = G$, as then the action $\beta$ defined by $\beta_t = \alpha_{t^{-1}}$ commutes with $\alpha$. In that case $f * x = x * f$ and $\Gbracket{x}{y} = \Hbracket{y}{x}$. 
    
    This example includes most setups of quantum harmonic analysis \cite{We84, LuSk18, FuGa25}, where operator shifts of abelian groups are considered. 
    Werner's setting is obtained for $G = \mathbb R^{2d}$, with the projective, square-integrable, irreducible representation $\pi$ of $G$ on $\H = L^2(\mathbb R^d)$: \[
        (\pi_{(t,\omega)} f)(s) = e^{2 \pi i \omega \cdot s} f(s - t),
        \quad f \in \H, (t,\omega) \in G.
    \] One defines the shift action $\alpha_g(T) = \pi_g T \pi_g^*$ for bounded operators $T \in \M = B(\H)$. It is invariant under the trace $\tau = \operatorname{Tr}$. With the parity operator $(Pf)(s) = f(-s), f \in \H$ one may put $\widetilde T = (P T P)^*, T \in \M$ to recover Werner's operator-operator convolution \[
        S * T = \Gbracket{S}{\widetilde T}.
    \]
    Since $\alpha_{g}(PTP) = P\alpha_{g^{-1}}(T)P$ it holds that $\Gbracket{S}{\widetilde T} = \Gbracket{T}{\widetilde S} = \Hbracket{\widetilde S}{T}$. So the associativity of the operator convolutions \eqref{eq:werners-operator-associativity}, that is \cite[Proposition 3.2(i)]{We84}, can be expressed as \[
        \Gbracket{S}{\widetilde T} * R
        = (S * T) * R
        = S * (T * R)
        = S * \Gbracket{T}{\widetilde R}
        = S * \Hbracket{\widetilde T}{R}.
    \] We may hence rewrite \eqref{eq:werners-operator-associativity} as the compatibility $\Gbracket{x}{y} * z = x * \Hbracket{y}{z}$ for all $x,y,z \in \mathfrak m_\tau$.
\end{example}

\begin{example}\label{ex:commutative-example}
    Consider a single unimodular group $G$ with bi-invariant Haar measure $m$. We may consider the abelian von Neumann algebra $\M = L^\infty(G, m)$ of multiplication operators on $L^2(G, m)$ by essentially bounded functions. It has the semifinite trace $\tau = \int_G \,\cdot\, \dif m$, and two canonical actions $\alpha_s = \lambda_s = f \mapsto f(s^{-1}\,\cdot\,)$ and $\beta_t = \rho_t = f \mapsto f(\,\cdot\, t)$. That is, $H = G$ admits commuting, $\tau$-preserving actions on $\M$. 
    Here $f_1 * x$ and $x * f_2$ are to ordinary function convolutions, and \[
        \Gbracket{x}{y} = x * \tilde y, \quad
        \Hbracket{y}{z} = \tilde y * z,
    \] with $\tilde y(s) = \overline{y(s^{-1})}, s \in G, y \in \M$. Associativity of convolutions states $
        \Gbracket{x}{y} * z 
        = x * \Hbracket{y}{z}.
    $
\end{example}

In general, it might not necessarily be the case that $\Gbracket{x}{y}$ or $\Hbracket{y}{z}$ are integrable for any choice of nonzero $x,y,z \in \mathfrak m_\tau$. We therefore want to assume that the commuting, $\tau$-preserving actions $\alpha, \beta$ satisfy the conditions of the Duflo--Moore theorem. 
The following result will show that only constant Duflo--Moore operators can appear in that situation. 
This limitation gives an explanation why e.g. in QHA over the non-unimodular affine group \cite{BeBeLuSk22} only left-sided function-operator convolutions were considered: Any commuting operator-shift defining compatible right-sided function-operator convolutions could not have satisfied Moyal's identity.

\begin{theorem}
    Let $G, H$ be locally compact groups, and let $\alpha,\beta$ be a pair of commuting and $\tau$-preserving actions on $\M$. 
    Then the following holds: \[
        \text{The actions $\alpha,\beta$ are ergodic and $\tau$-integrable} \implies \text{The groups $G,H$ are unimodular.}
    \]
\end{theorem}
\begin{proof}
    We assume the $\tau$-preserving action $\alpha$ of $G$ to be ergodic and $\tau$-integrable. Let $D$ be the positive, invertible Duflo--Moore operator affiliated to $\M$ for $\alpha$. Let $\beta$ be a $\tau$-preserving action of $H$ that commutes with $\alpha$. Consider nonzero $x,y \in \M_+$ with both $\tau(x)$ and $\tau_{D^{-1}}(y)$ finite. For all $s \in G, t \in H$ it holds \[
        \tau(x \, \alpha_s(\beta_{t^{-1}}(y))) 
        = \tau(x \, \beta_{t^{-1}}(\alpha_s(y)) ) 
        = \tau(\beta_t(x) \, \alpha_s(y)).
    \] Hence, by \cref{thm:duflo-moore} it follows for all $t \in H$ that \[
        \int_G \tau(x \, \alpha_s(\beta_{t^{-1}}(y))) \dif m_G(s) 
        = \int_G \tau(\beta_t(x) \, \alpha_s(y)) \dif m_G(s)
        = \tau(\beta_t(x)) \tau_{D^{-1}}(y) 
        = \tau(x) \tau_{D^{-1}}(y).
    \] Consequently, all $\beta_{t^{-1}}(y), t \in H$ are admissible elements of $\M$ and satisfy \[
        \tau(x) \tau_{D^{-1}}(\beta_{t^{-1}}(y)) 
        = \tau(x) \tau_{D^{-1}}(y).
    \] This shows that $\tau_{D^{-1}}$ is $\beta$-invariant and $\beta_t(D^{-1}) = D^{-1}$. So, assuming ergodicity for $\beta$ implies that $D$ is a scalar multiple of the identity, and hence that $G$ is unimodular. Interchanging the roles of $\alpha$ and $\beta$ shows the claim.
\end{proof}

From now on, we will restrict ourselves to \emph{unimodular} groups $G$ and $H$ that act via commuting, $\tau$-preserving actions $\alpha$ and $\beta$ on a common von Neumann algebra $\M$, which are both ergodic and $\tau$-integrable. We will omit writing the bi-invariant Haar measures $m_G$ and $m_H$, and fix the constants $d_\alpha, d_\beta > 0$ such that $d_\alpha 1$ and $d_\beta 1$ are the bounded Duflo--Moore operators for the actions $\alpha$ and $\beta$ respectively. Whenever $1/p + 1/q = 1 + 1/r$, we have sesquilinear maps \[
    \Gbracket{\,\cdot\,}{\,\cdot\,} \colon L^p(\M) \times L^q(\M) \to L^r(G), \quad
    \Hbracket{\,\cdot\,}{\,\cdot\,} \colon L^q(\M) \times L^p(\M) \to L^r(H),
\] such that for $x,y,z \in \mathfrak m_\tau$ we have identifications with continuous, bounded functions as in \eqref{eq:brackets-definition}. 
Moreover, we may consider the extended function-operator convolutions of \cref{sect:convolutions}. In particular, whenever $1 \leq a,b,c \leq r \leq \infty$ are such that $1/a + 1/b + 1/c = 2 + 1/r$, for all $x \in L^a(\M), y \in L^b(\M)$ and $z \in L^c(\M)$ we have defined two elements in $L^r(\M)$: \[
    \Gbracket{x}{y} * z \qquad \text{and}\qquad 
    x * \Hbracket{y}{z}.
\] The remainder of this section will be devoted to showing their equality whenever $d_\alpha = d_\beta$.


\begin{lemma}\label{lemma:two-variable-function}
    Let $x,y,z,w \in \mathfrak m_\tau$. Define \[
        f(s,t) 
        = \tau( w \beta_t(x) \alpha_s( \beta_t(y^*) z ) ),
        \quad s \in G, t \in H.
    \] Then $f$ is continuous and integrable as a function on $G \times H$.
\end{lemma}
\begin{proof}
    We freely assume all $x,y,z,w$ are nonzero. We show continuity in $(s_0, t_0) \in G \times H$. Let $\epsilon > 0$. By the ultraweak continuity of $\alpha$ we find an open neighborhood $U \subseteq G$ of $s_0$ such that \[
        |\tau(w \beta_{t_0}(x) \left(\alpha_{s_0}(\beta_{t_0}(y^*) z) - \alpha_s(\beta_{t_0}(y^*) z) \right))| < \eps, \quad s \in U.
    \] 
    Since $\alpha$ preserves the trace $\tau$, we know that it extends to a strongly continuous action by isometries on $L^2(\M)$.
    So we may choose an open neighborhood $V \subseteq H$ of $t_0$ such that for all $t \in V$ we have both \[
        \|\beta_{t_0}(x) - \beta_t(x)\|_2 < \eps/(\|y\|_2 \|z\|_\infty \|w\|_\infty),
    \] \[
        \|\beta_{t_0}(y) - \beta_t(y)\|_2 < \eps/(\|x\|_2 \|z\|_\infty \|w\|_\infty).
    \] For all $(s,t) \in U \times V$ it follows \[
        |\tau(w \left(\beta_{t_0}(x) - \beta_t(x)\right) \alpha_s(\beta_{t_0}(y)^* z))| \leq 
        \|\beta_{t_0}(x) - \beta_t(x)\|_2 \|\alpha_s(\beta_{t_0}(y)^* z) w\|_2 < \eps, 
    \]\[
        |\tau(w \beta_t(x) \alpha_s( \left(\beta_{t_0}(y) - \beta_t(y) \right)^* z))| 
        \leq \|z \alpha_{s^{-1}}(w \beta_t(x))\|_2 \|\beta_{t_0}(y) - \beta_t(y)\|_2 < \eps.
    \] In combination, \[
        |\tau(w \beta_{t_0}(x) \alpha_{s_0}(\beta_{t_0}(y^*) z)) - \tau(w \beta_t(x) \alpha_s(\beta_t(y^*) z))| 
        < 3\eps, \quad (s,t) \in U \times V.
    \]

    Now, as $f$ is continuous and hence measurable, by Tonelli's theorem it will suffice to show that the iterated integral $\int_H \int_G |f| \dif m_G \dif m_H$ is finite. As $f(\,\cdot\,, t) = \Gbracket{w \beta_t(x)}{(\beta_t(y^*) z)^*}$, we compute \[
        \int_G |f(s,t)| \dif s \leq d_\alpha^{-1} \|w \beta_t(x)\|_1 \|\beta_t(y^*) z\|_1, \quad t \in H.
    \] Let $w = u |w|$ be the polar decomposition with partial isometry $u \in \M$. Then for all $t \in H$: \[
        \|w \beta_t(x)\|_1
        = \| |w^*|^{1/2} u \, |w|^{1/2} \beta_t(x) \|_1
        \leq \||w^*|^{1/2} u\|_2 \||w|^{1/2} \beta_t(x)\|_2
        = \|w\|_1^{1/2} \Hbracket{|w|}{x x^*}(t)^{1/2}.
    \] Analogously, $\|\beta_t(y^*) z\|_1 \leq \|z\|_1^{1/2} \Hbracket{|z|}{y^* y}(t)^{1/2}$ for all $t \in H$. Therefore, by the Cauchy-Schwarz inequality and two more applications of the Duflo--Moore theorem \[\begin{aligned}
        \int_H \int_G |f(s,t)| \dif s \dif t
        &\leq d_\alpha^{-1} \|w\|_1^{1/2} \|z\|_1^{1/2} (\int_H \Hbracket{|w|}{x x^*} \dif m_H )^{1/2} (\int_H \Hbracket{|z|}{y^* y} \dif m_H)^{1/2} \\
        &= d_\alpha^{-1} \|w\|_1^{1/2} \|z\|_1^{1/2} (d_\beta^{-1} \|w\|_1 \|x x^*\|_1)^{1/2} (d_\beta^{-1} \|z\|_1 \|y^* y\|_1)^{1/2} < \infty.
    \end{aligned}\]
\end{proof}

\begin{lemma}\label{lemma:associativity}
    Let $x,y,z \in \mathfrak m_\tau$. Then as elements of $\mathfrak m_\tau \subseteq \M$ we have \[
        d_\alpha \cdot \Gbracket{x}{y} * z
        = d_\beta \cdot x * \Hbracket{y}{z}.
    \]
\end{lemma}
\begin{proof}
    As $\tau$ is semifinite, it suffices to show for all $w \in \mathfrak m_\tau$ that \[
        d_\beta^{-1} \, \tau( \left( \Gbracket{x}{y} * z \right) w )
        = d_\alpha^{-1} \, \tau( w \left( x * \Hbracket{y}{z} \right) ).
    \]
    
    First, using that $\alpha$ and $\beta$ commute, we find for all $s \in G$ and $t \in H$ the equality: \[
        \tau( \beta_t( x \alpha_s(y^*) ) \alpha_s(z) w )
        = \tau( w \beta_t(x) \alpha_s( \beta_t(y^*) z ) ).
    \]
    In \cref{lemma:two-variable-function} we saw that the above defines an integrable function $f$ on $G \times H$. Using Fubini's theorem, we may integrate $f$ in two ways: We use the Duflo--Moore theorem for $\beta$ to find \[
        \int_G \int_H f(s,t) \dif t \dif s
        = \int_G d_\beta^{-1} \tau(x \alpha_s(y^*)) \tau(\alpha_s(z) w) \dif s
        = d_\beta^{-1} \tau(\left( \Gbracket{x}{y} * z \right) w).
    \] On the other hand, we use the Duflo--Moore theorem for $\alpha$ to find \[
        \int_H \int_G f(s,t) \dif s \dif t
        = \int_H d_\alpha^{-1} \tau(w \beta_t(x)) \tau(\beta_t(y^*) z) \dif t
        = d_\alpha^{-1} \tau(w \left(x * \Hbracket{y}{z}\right)).
    \]
\end{proof}

\begin{theorem}\label{thm:associativity}
    Let $\alpha$, $\beta$ be commuting, $\tau$-preserving actions of unimodular groups $G$, $H$ on the same von Neumann algebra $\M$ with semifinite trace $\tau$, which are ergodic and $\tau$-integrable. Let $d_\alpha, d_\beta > 0$ be such that $d_\alpha 1$ and $d_\beta 1$ are their bounded Duflo--Moore operators.
    
    Let $1 \leq a,b,c \leq r \leq \infty$ be such that $1/a + 1/b + 1/c = 2 + 1/r$. Then \begin{equation}\label{eq:tripple-convolution}
        d_\alpha \cdot \Gbracket{x}{y} * z
        = d_\beta \cdot x * \Hbracket{y}{z}
    \end{equation}
    as elements of $L^r(\M)$ for all triples of $x \in L^a(\M)$, $y \in L^b(\M)$ and $z \in L^c(\M)$.
\end{theorem}

\begin{proof}
    We choose $1 \leq s,t \leq r$ are such that $1 + 1/s = 1/a + 1/b$ and $1 + 1/t = 1/b + 1/c$. Then $\Gbracket{x}{y} \in L^{s}(G)$ and $\Hbracket{y}{z} \in L^{t}(H)$, and since $1 + 1/r = 1/s + 1/c = 1/a + 1/t$ both sides of the claimed equality \eqref{eq:tripple-convolution} are well defined elements of $L^r(\M)$. We will show that their difference is zero in the norm of $L^r(\M)$.

    Let $(x_k)_k, (y_m)_m, (z_n)_n \subseteq \mathfrak m_\tau$ be such that $x_k \to x$ in $L^a(\M)$, that $y_m \to y$ in $L^b(\M)$ and that $z_n \to z$ in $L^c(\M)$. From \cref{lemma:associativity} it follows for all indices $k,m,n$ that \[
        d_\alpha \cdot \Gbracket{x_k}{y_m} * z_n 
        = d_\beta \cdot x_k * \Hbracket{y_m}{z_n}.
    \] 
    From \eqref{eq:bracket-continuous} it follows that $\Gbracket{\,\cdot\,}{\,\cdot\,} \colon L^a(\M) \times L^b(\M) \to L^s(G)$ is continuous, and also the function-operator convolution $L^s(G) \times L^c(\M) \to L^r(\M)$ is continuous. In combination, we find that $\Gbracket{x_k}{y_m} * z_n \to \Gbracket{x}{y} * z$ in $L^r(\M)$.
    Analogously, $x_k * \Hbracket{y_m}{z_n} \to x * \Hbracket{y}{z}$ in $L^r(\M)$. So \[\begin{aligned}
        &\| d_\alpha \cdot \Gbracket{x}{y} * z - d_\beta \cdot x * \Hbracket{y}{z} \|_r \\
        &\qquad \leq d_\alpha \cdot \| \Gbracket{x}{y} * z - \Gbracket{x_k}{y_m} * z_n \|_r + d_\beta \cdot \| x_k * \Hbracket{y_m}{z_n} - x * \Hbracket{y}{z} \|_r \to 0.
    \end{aligned}\] 
\end{proof}

\begin{remark}
    Let $\alpha,\beta$ be again only $\tau$-preserving actions of locally compact groups $G, H$ on the common von Neumann algebra $\M$ with semifinite trace $\tau$. If only one action, say $\beta$, satisfies the conditions of the Duflo--Moore theorem with operator $D$, one can still consider a weak version of the bracket compatibility: 
    For the weights $\psi = \tau_w, w \in \mathfrak m_\tau^+$ it may hold \begin{equation*}
        \int_G \Gbracket{x}{y}(s) \, \psi(\alpha_s(z)) \dif m_G(s) = \psi\left(x * \Hbracket{y}{z}\right), \quad x,z \in \mathfrak m_\tau^+, y \in \mathfrak m_{\tau_{D^{-1}}}^+.
    \end{equation*}
    From this assertion it follows that also $\alpha$ has to satisfy the conditions of the Duflo--Moore theorem, and that the Duflo--Moore operators of $\alpha$ and $\beta$ agree:

    As $\tau$ is semifinite, one finds a net of projections $(p_i)_i \subseteq \mathfrak m_\tau$ with least upper bound $1 \in \M_+$. Then considering $\psi = \tau_{p_i}$ and applying the monotone convergence theorem one finds \[
        \int_G \Gbracket{x}{y}(s) \, \tau(z) \dif m_G(s) = \tau\left(x * \Hbracket{y}{z}\right) = \tau(x) \int_H \Hbracket{y}{z} \dif\widecheck m_H.
    \]
    After applying the Duflo--Moore theorem for $\beta$ and canceling factors we have \[
        \int_G \Gbracket{x}{y}(s) \dif m_G(s) = \tau(x) \tau_{D^{-1}}(y), \quad x \in \mathfrak m_\tau^+, y \in \mathfrak m_{\tau_{D^{-1}}}^+.
    \]
    Then \cref{prop:reverse-duflo-moore} shows that $\alpha$ is ergodic, $\tau$-integrable and has Duflo--Moore operator $D$.
\end{remark}

\subsection{Bimodule structures}\label{sect:bimodule}

Let $G$ and $H$ be unimodular groups. Let $\M$ be a von Neumann algebra with semifinite trace $\tau$, and let $\alpha, \beta$ be commuting, $\tau$-preserving actions on $\M$. We saw that $\mathfrak m_\tau$ carries a left-$L^1(G)$ and right-$L^1(H)$ bimodule structure via function-operator convolutions. Assuming that $\alpha, \beta$ are both ergodic and $\tau$-integrable, we have defined sesquilinear brackets \[
    \Gbracket{\,\cdot\,}{\,\cdot\,}\colon \mathfrak m_\tau \times \mathfrak m_\tau \to L^1(G), \quad
    \Hbracket{\,\cdot\,}{\,\cdot\,}\colon \mathfrak m_\tau \times \mathfrak m_\tau \to L^1(H),
\]
We normalize the Haar measures on $G$ and $H$ such that it holds $d_\alpha = d_\beta$ for the constants $d_\alpha, d_\beta > 0$ from the Duflo--Moore theorem. 

Denote by $\lambda$ and $\rho$ the integrated left- and right-regular representations of $L^1(G)$ on $L^2(G)$ and respectively $L^1(H)$ on $L^2(H)$ via left- and right-convolution. The norm-closures of $\lambda(L^1(G))$ and $\rho(L^1(H))$ define the reduced group $C^*$-algebras $C_r^*(G)$ and $C_r^*(H)$. 

\begin{proposition}\label{prop:morita-equivalence}
    Suppose that $\lambda(\Gbracket{\mathfrak m_\tau}{\mathfrak m_\tau}) \subseteq C_r^*(G)$ and $\rho(\Hbracket{\mathfrak m_\tau}{\mathfrak m_\tau}) \subseteq C_r^*(H)$ span dense ideals of the reduced group $C^*$-algebras.
    Then $C_r^*(G)$ and $C_r^*(H)$ are Morita equivalent.
\end{proposition}
\begin{proof}
    We will show that $\mathfrak m_\tau$ is a pre-imprimitivity $L^1(G)$-$L^1(H)$-bimodule in the sense of \cite[Definition 3.9]{WiRae98}. Then by \cite[Proposition 3.12]{WiRae98} the completion of $\mathfrak m_\tau$ with respect to the norm given by \[
        \|x\| = \|\lambda(\Gbracket{x}{x})\|^{1/2} = \|\rho(\Hbracket{x}{x})\|^{1/2}, \quad x \in \mathfrak m_\tau,
    \] will be an imprimitivity bimodule for the reduced $C^*$-algebras.

    The bimodule structure on $\mathfrak m_\tau$ is given by the function-operator convolution. In addition, $\Gbracket{\,\cdot\,}{\,\cdot\,}$ is a left-$L^1(G)$- and $\Hbracket{\,\cdot\,}{\,\cdot\,}$ a right-$L^1(H)$-pre-inner product for the vector space $\mathfrak m_\tau$: For this we 
    compute that for all $x \in \mathfrak m_\tau$ and all $g \in L^2(G), h \in L^2(H)$ it holds \begin{equation}\label{eq:positive-bracket}
        \langle \Gbracket{x}{x} * g, g \rangle
        = \|g^* * x\|_2^2 \geq 0, \quad 
        \langle h * \Hbracket{x}{x}, h \rangle
        = \|x * h^*\|_2^2 \geq 0.
    \end{equation}
    So, for all $x \in \mathfrak m_\tau$, $\lambda(\Gbracket{x}{x})$ defines a positive element of $C_r^*(G)$, and $\rho(\Hbracket{x}{x})$ a positive element of $C_r^*(H)$. 
    
    For the pre-imprimitivity bimodule structure, all it remains to check is \[
        \Hbracket{f * x}{f * x} \leq \|\lambda(f)\|^2 \, \Hbracket{x}{x}, \quad x \in \mathfrak m_\tau, f \in L^1(G),
    \] in the sense of positivity of the $C^*$-algebra $C_r^*(H)$. By an analogous argument, $\Gbracket{x * f}{x * f} \leq \|\rho(f)\|^2 \, \Gbracket{x}{x}$ will follow in $C_r^*(G)$. Using \eqref{eq:positive-bracket} we may equivalently show for all $h \in L^2(H)$ \[
        \|(f * x) * h^*\|_2^2
        \leq 
        \|\lambda(f)\|^2 \|x * h^*\|_2^2,
        \quad x \in \mathfrak m_\tau, f \in L^1(G).
    \]
    This indeed holds because of \cref{lemma:weak-containment} and $(f * x) * h^* = f * (x * h^*)$.
\end{proof}

\begin{remark}
    The density condition in \cref{prop:morita-equivalence} is not automatic: For a compact group $G$ consider the trivial action $\alpha_g = \operatorname{id}, g \in G$ on the one-dimensional von Neumann algebra $\M = \mathbb C$ with finite trace $\tau = \operatorname{id}$. Then $\alpha$ is ergodic, $\tau$-preserving and $\tau$-integrable. However, $\Gbracket{x}{y}$ are the constant functions with value $x \overline y$, for all $x,y \in \M$.
\end{remark}

\subsection{Regularity}\label{sect:regularity}

We again fix commuting, $\tau$-preserving actions $\alpha,\beta$ of unimodular groups $G,H$ on a common von Neumann algebra $\M$ with semifinite trace $\tau$. We assume that both actions are ergodic and $\tau$-integrable, such that the Duflo--Moore theorem applies and yields constants $d_\alpha, d_\beta > 0$. We assume $d_\alpha = d_\beta = 1$ for simplicity.

For a fixed operator $a \in L^1(\M)$ we can consider the following statements: 

\begin{enumerate}
    \item[(1)] The left-translates of $\Gbracket{a}{a} \in L^1(G)$ span a dense subset of $L^1(G)$.
    
        \item[(2a)] The map $L^1(\M) \to L^1(G), x \mapsto \Gbracket{x}{a}$ has dense range.
        \item[(2b)] The map $L^\infty(G) \to \M, f \mapsto f * a$ is injective.

        \item[(3a)] The map $\M \to C_b(G), y \mapsto \Gbracket{a}{y}$ is injective.
        \item[(3b)] The $\alpha$-translates of $a$ span a dense subset of $L^1(\M)$.
\end{enumerate}

\begin{enumerate}
    \item[(1')] The right-translates of $\Hbracket{a}{a} \in L^1(H)$ span a dense subset of $L^1(H)$.

        \item[(2'a)] The map $L^1(\M) \to L^1(H), z \mapsto \Hbracket{a}{z}$ has dense range.
        \item[(2'b)] The map $L^\infty(H) \to \M, f \mapsto a * f$ is injective.
    
        \item[(3'a)] The map $\M \to C_b(H), y \mapsto \Hbracket{y}{a}$ is injective.
        \item[(3'b)] The $\beta$-translates of $a$ span a dense subset of $L^1(\M)$.
\end{enumerate}

As a consequence of the Hahn-Banach theorem and $\M$ being the dual of $L^1(\M)$, we have equivalences of the statements $(2a) \iff (2b)$ and $(3a) \iff (3b)$, as well as $(2'a) \iff (2'b)$ and $(3'a) \iff (3'b)$. So, for $X \in \{2, 3, 2', 3'\}$ we write $(X)$ whenever $(Xa)$ and $(Xb)$ hold.

\begin{remark}
    Recall \cref{ex:qha-example} with $G = H$ an abelian group. Then for every $n \in \{1,2,3\}$ statements $(n)$ and $(n')$ are identical. In the setting of quantum harmonic analysis, where $\M = B(\mathcal H)$ is considered with the standard trace $\tau = \operatorname{Tr}$, and $\alpha$ is the conjugation by some irreducible, square-integrable (projective) representation $\pi$ of an abelian $G$, statements of above form are considered and shown to be equivalent already in \cite[Proposition 3.5]{We84}, \cite{KiLaSchuWer12}.
    
    In \cite{KiLaSchuWer12} they emphasize the application to informational completeness of shift-covariant observables: The measurement of a quantum state $\rho \in L^1(\M)$ with respect to a $\M_+$-valued measure $\mu$ on $G$ is modeled with the probability density $f_\rho(s) = \operatorname{Tr}(\rho \, \mu(s))$. If the observable is shift-covariant, meaning $\mu(s A) = \alpha_s(\mu(A))$ for $s \in G, A \subseteq G$, it can be shown that $\mu(A) = \int_A \alpha_s(a) \dif m(s)$ for some positive $a \in L^1(\M)$ with $\tau(a) = 1$. Then $f_\rho = \Gbracket{\rho}{a}$. That is, injectivity of $\rho \mapsto f_\rho$ characterizes reconstructability of quantum states from their measurements.
\end{remark}

\begin{lemma}\label{lemma:regularity-implications}
    Regarding the above statements:
    \begin{enumerate}
        \item[(i)] For all $a \in L^1(\M)$ it holds (1) $\implies$ (2) $\implies$ (3'), as well as (1') $\implies$ (2') $\implies$ (3).
        \item[(ii)] Suppose the statements (2) and (2') are satisfied by some nonzero $a_1,a_2 \in L^1(\M)$. Then for all $a \in L^1(\M)$ it holds (3') $\implies$ (2), as well as (3) $\implies$ (2').
        \item[(iii)] For all $a \in L^1(\M)$ it holds $((2) \text{ and } (3)) \implies (1)$, as well as $((2') \text{ and } (3')) \implies (1')$.
    \end{enumerate}
\end{lemma}
\begin{proof}
    (1) to (2a): We find that already a subset of $\Gbracket{L^1(\M)}{a}$ is dense in $L^1(G)$: Namely \[
        \Gbracket{ \span\{ \alpha_s(a) \mid s \in G \} }{a}
        = \span\{ \Gbracket{ \alpha_s(a)}{a} \mid s \in G \}
        = \span\{ \lambda_s(\Gbracket{a}{a}) \mid s \in G \}
    \] is dense in $L^1(G)$ by the assumption. The implication (1') to (2'a) is analogous.

    (2b) to (3'a): Suppose $\Hbracket{y}{a} = 0$ for some $y \in \M$. Then for all $x \in L^1(\M)$ by \cref{thm:associativity}, \[
        0 = x * \Hbracket{y}{a} = \Gbracket{x}{y} * a.
    \] By (2b) this implies $0 = \Gbracket{x}{y} \in C_b(G) \subseteq L^\infty(G)$ for all $x \in L^1(\M)$. By continuity we may evaluate $0 = \Gbracket{x}{y}(e) = \tau(x y^*)$ for all $x \in L^1(\M)$. Hence $y = 0$. The implication (2'b) to (3a) is analogous.

    (3'a) to (2b) under the assumption that some $a_0 \in L^1(\M)$ satisfies (2a): Suppose $f \in L^\infty(G)$ satisfies $f * a = 0$. Then for all $b \in L^1(\M)$ 
    we find \[
        0 = \Hbracket{b}{f * a} = \Hbracket{f^* * b}{a}.
    \] Using (3'a) this implies $f^* * b = 0$ for all $b \in L^1(\M)$. Since $\Gbracket{L^1(\M)}{a_0} \subseteq L^1(G)$ is dense, \[
        0 = \Gbracket{f^* * b}{a_0} 
        = f^* * \Gbracket{b}{a_0}
    \] forces $f^* = 0$. The implication (3a) to (2'b) is analogous.

    (2a) and (3b) together imply (1): Let $f \in L^1(G)$ and $\eps > 0$ be arbitrary, and use (2a) to find $x \in L^1(\M)$ such that $\|\Gbracket{x}{a} - f\|_1 < \eps$. By (3b) we find $g_i \in G$ and $c_i \in \C$ such that $\|\sum c_i \alpha_{g_i}(a) - x\|_1 < \eps / \|a\|_1$. Since \[
        \|\sum c_i \Gbracket{\alpha_{g_i}(a)}{a} - \Gbracket{x}{a}\|_1
        \leq \|\sum c_i \alpha_{g_i}(a) - x\|_1 \|a\|_1
        < \eps,
    \] it holds \[
        \| \sum c_i \Gbracket{\alpha_{g_i}(a)}{a} - f \|_1 
        < 2\eps.
    \] The implication of (2'a) and (3'b) together to (1') is analogous.
\end{proof}

We have to refine the notion of \emph{regularity}, introduced in \cite{KiLaSchuWer12}, to avoid ambiguity. An element $a \in L^1(\M)$ we call \emph{$\alpha$-regular} if the set of translates $\alpha_g(a), g \in G$ spans a dense subset of $L^1(\M)$. We use the same language for $\beta$ on $H$ instead of $\alpha$ on $G$. Similarly, a function $f \in L^1(G)$ we call \emph{$\lambda$-regular} if the set of translates $\lambda_g(f), g \in G$ spans a dense subset of $L^1(G)$. We use the same language with $\rho$ for functions on $H$ instead of $\lambda$ for functions on $G$.

We combine statements $(n)$ and $(n')$ for $n \in \{1,2,3\}$ as follows:

\begin{theorem}\label{thm:regularity}
    Suppose there exists $a_1, a_2 \in L^1(\M)$ such that $\Gbracket{L^1(\M)}{a_1} \subseteq L^1(G)$ is dense and $\Hbracket{a_2}{L^1(\M)} \subseteq L^1(H)$ is dense. Then for all $a \in L^1(\M)$ the following are equivalent: \begin{enumerate}
        \item[(i)] It holds that $\Gbracket{a}{a} \in L^1(G)$ is $\lambda$-regular and that $\Hbracket{a}{a} \in L^1(H)$ is $\rho$-regular.
        \item[(ii)] Both the sets $\Gbracket{L^1(\M)}{a} \subseteq L^1(G)$ and $\Hbracket{a}{L^1(\M)} \subseteq L^1(H)$ are dense.
        \item[(iii)] The operator $a \in L^1(\M)$ is both $\alpha$-regular and $\beta$-regular. 
    \end{enumerate}
\end{theorem}
\begin{proof}
    In \cref{lemma:regularity-implications} we have seen that in general it holds $(i) \iff (ii) \implies (iii)$, since \[
        ((1)\text{ and }(1')) \implies ((2)\text{ and }(2')) \implies ((3)\text{ and }(3')),
    \]
    \[
        ((2)\text{ and }(2')) \implies ((2)\text{ and }(2') \text{ and } (3)\text{ and }(3')) \implies ((1)\text{ and }(1')).
    \] Furthermore, if there exist some elements satisfying $(2)$ and $(2')$, then \[
        ((3)\text{ and }(3')) \implies ((2)\text{ and }(2')).
    \]
\end{proof}

In the case of quantum harmonic analysis with the action given through conjugation with the irreducible, projective Weyl-representation of a position-momentum phase-space, it can be shown that there exists a regular element satisfying $(1)$ and $(1')$, and then \cref{thm:regularity} recovers \cite[Proposition 1]{KiLaSchuWer12}. We do not formulate statements regarding \emph{$p$-regularity}, but they can be shown by a similar line of arguments.

We want to remark that also \cite[Theorem 5.29]{Ha23} with $p = 1$ states a special case of the above, where we however believe to be a gap in their proof, due to their lack of the operator-convolution associativity. We crucially had to use \cref{thm:associativity}.

\subsection{A Tauberian theorem}\label{sect:tauberian}

As an application of the equivalent formulations of regularity of the previous chapter, one may show an adaptation of a classical theorem of Wiener to bounded operators as seen in \cite[Theorem 5.1]{LuSk21}. 
For this section we again fix two commuting, $\tau$-preserving actions $\alpha,\beta$ of unimodular groups $G,H$ on a common von Neumann algebra $\M$. We use the same notation as in \cite{LuSk21} and denote \[
    W = \{ a \in L^1(\M) \mid a \text{ is $\alpha$-regular and $\beta$-regular} \}.
\]
Also, we want to assume throughout that there exists operators $a_1, a_2 \in L^1(\M)$, such that the equivalences of \cref{thm:regularity} hold for all $a \in L^1(\M)$, so in particular $W \neq \emptyset$.

\begin{lemma}
    It holds that $\Gbracket{W}{W} * W \subseteq W$.
\end{lemma}
\begin{proof}
    Both $L^1(\M) \to L^1(G), x \mapsto \Gbracket{x}{w}$ and $L^1(G) \to L^1(\M), f \mapsto f * w$ are continuous for $w \in L^1(\M)$. So for regular $x,y,z \in L^1(\M)$ the claim follows from the density in $L^1(\M)$ of \[
        \operatorname{span}\{ \alpha_g( \Gbracket{x}{y} * z ) \mid g \in G \}
        = \Gbracket{ \operatorname{span}\{ \alpha_g(x) \mid g \in G \} }{y} * z.
    \] Since $\Gbracket{x}{y} * z = x * \Hbracket{y}{z}$, we can use the analogous argument for $\beta$ translates.
\end{proof}

For locally compact Hausdorff spaces $X$, recall that $C_0(X)$ denotes the set of continuous functions on $X$ that vanish at infinity. Recall that $C_0(X)$ is $\|\cdot\|_\infty$-norm-closed inside $C_b(X)$. This means in particular, that for $f \in C_b(X)$ the following holds: If for every $\epsilon > 0$ there exist some $h_\epsilon \in C_0(X)$ such that $\|f - h_\epsilon\|_\infty < \epsilon$, then it follows that $f \in C_0(X)$.

Consider \[
    K = \overline{C_0(G) * L^1(\M)}^{\|\cdot\|_\M} \subseteq \M. 
\] 
Note that for $\M = L^\infty(G), \tau = \int_G \,\cdot\, \dif m$ and $\alpha = \lambda$ we recover $K = C_0(G)$, since in that setting $L^1(\M) = L^1(G)$ leaves $C_0(G)$ invariant under convolution and admits approximate identities. For $\M = B(\mathcal H), \tau = \operatorname{Tr}$ and $\alpha$ conjugation by the phase space representation of quantum harmonic analysis, it is contained in the compact operators \cite[Lemma 2.3]{LuSk21}.

\begin{lemma}
    It holds that $\Gbracket{K}{L^1(\M)} \subseteq C_0(G)$.
\end{lemma}
\begin{proof}
    Let $z \in K$ and $y \in L^1(\M)$. If $y = 0$ clearly $\Gbracket{z}{y} = 0 \in C_0(G)$. Otherwise let $\epsilon > 0$. We find $z_\epsilon = f_\epsilon * x_\epsilon \in \M$, where $f_\epsilon \in C_0(G)$ and $x_\epsilon \in L^1(\M)$, such that $\|z - z_\epsilon\|_\infty < \epsilon / \|y\|_1$. Then \[
        h_\epsilon = \Gbracket{z_\epsilon}{y} = f_\epsilon * \Gbracket{x_\epsilon}{y} \in C_0(G) * L^1(G) \subseteq C_0(G).
    \] satisfies $\|\Gbracket{z}{y} - h_\eps\|_\infty = \|\Gbracket{z - z_\eps}{y}\|_\infty \leq \|z - z_\eps\|_\infty \|y\|_1 < \eps$. Then from the before recalled fact, it follows $\Gbracket{z}{y} \in C_0(G)$.
\end{proof}

We are now ready to state a theorem that includes Wiener's classical Tauberian theorem when applied for $\M = L^\infty(G)$, and that mirrors the Tauberian theorem for bounded operators of \cite[Theorem 5.1]{LuSk21} when applied for $\M = B(\H)$. In the general case it is however unclear to the author whether $\Hbracket{W}{W}$ can be a proper subset of all $\rho$-regular functions.

\begin{theorem}
    Let $\alpha,\beta \colon G,H \curvearrowright \M$ as above. Let $x \in \M$. The following are equivalent: \begin{enumerate}
        \item[(i)] There exists some $y \in W$ such that $\Gbracket{x}{y} \in C_0(G)$.
        \item[(ii)] There exists some $f \in \Hbracket{W}{W}$ such that $x * f \in K$.
    \end{enumerate}
    If either of the above hold, then both of the following hold: \begin{enumerate}
        \item[(1)] For all $y \in L^1(\M)$ it holds $\Gbracket{x}{y} \in C_0(G)$.
        \item[(2)] For all $f \in L^1(H)$ it holds $x * f \in K$.
    \end{enumerate}
\end{theorem}

\begin{proof}
    For $x = 0$ there is nothing to show, so let $x \neq 0$ throughout.

    (i) to (ii): Suppose $y \in W$ satisfies $\Gbracket{x}{y} \in C_0(G)$. Choose any $a \in W$ and set $f = \Hbracket{y}{a} \in \Hbracket{W}{W}$. It satisfies \[
        x * f = x * \Hbracket{y}{a}
        = \Gbracket{x}{y} * a
        \in C_0(G) * W \subseteq C_0(G) * L^1(M) \subseteq K.
    \]

    (ii) to (i): Suppose $f \in \Hbracket{W}{W}$ satisfies $x * f \in K$. Still it holds $f^* \in \Hbracket{W}{W}$. Choose any $a \in W$ and set $y = a * f^* \in W * \Hbracket{W}{W} = \Gbracket{W}{W} * W \subseteq W$. It satisfies \[
        \Gbracket{x}{y} = \Gbracket{x}{a * f^*}
        = \Gbracket{x * f}{a}
        \in \Gbracket{K}{W} \subseteq \Gbracket{K}{L^1(M)} \subseteq C_0(G).
    \]

    (i) to (1): Suppose $a \in W$ satisfies $\Gbracket{x}{a} \in C_0(G)$. Let $y \in L^1(\M)$ and $\eps > 0$. We know that $L^1(G) * a$ is dense in $L^1(\M)$. We find $f_\eps \in L^1(G)$ such that $\|f_\eps * a - y\|_1 < \eps/\|x\|_\M$. 
    Denote \[
        h_\eps = \Gbracket{x}{f_\eps * a} = \Gbracket{x}{a} * f^*_\eps \in C_0(G) * L^1(G) \subseteq C_0(G).
    \]
    Since is $\Gbracket{x}{y} \in C_b(G)$ the claim follows as \[
        \|h_\eps - \Gbracket{x}{y}\|_\infty
        = \|\Gbracket{x}{f_\eps * a - y}\|_\infty
        \leq \|x\|_\infty \|f_\eps * a - y\|_1 < \eps.
    \]

    (i) to (2): Suppose $y \in W$ satisfies $\Gbracket{x}{y} \in C_0(G)$. Let $f \in L^1(H)$ and $\eps > 0$. We know that $\Hbracket{y}{L^1(\M)}$ is dense in $L^1(H)$. We find $x_\eps \in L^1(\M)$ with $\|f - \Hbracket{y}{x_\eps}\|_1 < \eps / \|x\|_\infty$. Denote \[
        k_\eps = x * \Hbracket{y}{x_\eps} = \Gbracket{x}{y} * x_\eps \in C_0(G) * L^1(\M) \subseteq K.
    \] We verify $x * f \in K$ as \[
        \|x * f - k_\eps\|_\infty = \|x * (f - \Hbracket{y}{x_\eps})\|_\infty \leq \|x\|_\infty \|f - \Hbracket{y}{x_\eps}\|_1 < \eps.
    \]
\end{proof}

\appendix\section[Appendix]{Duflo-Moore assumptions}\label{sect:appendix}

In \cite{EnWe25} it is made clear that the conditions on a $\tau$-preserving action in \cref{thm:duflo-moore} are the correct ones to generalize the theorem of Duflo and Moore \cite{DuMo1976} for representations of locally compact groups for the case $\M = B(\mathcal H)$ with $\tau = \operatorname{Tr}$. Clearly $\tau$-invariance is necessary to even isometrically extend the automorphisms $\alpha$ to the $L^p(\M, \tau)$ spaces. It turns out, that the other two conditions on $\tau$-invariant actions can in general not be weakened:

\begin{proposition}\label{prop:reverse-duflo-moore}
    Let $\alpha$ be a $\tau$-preserving action of a locally compact group $G$ on a von Neumann algebra $\M$ with semifinite trace $\tau$. Then the following are equivalent: \begin{enumerate}
        \item The action $\alpha$ is ergodic and $\tau$-integrable.
        \item There exists a positive, invertible operator $D$ affiliated to $\M$ such that \eqref{eqn:duflo-moore} holds.
    \end{enumerate}
\end{proposition}
\begin{proof}
    The implication (1) to (2) is \cref{thm:duflo-moore}. Also, (2) implies $\tau$-integrability. We will now assume that some positive, invertible operator $D$ affiliated to $\M$ satisfies \eqref{eqn:duflo-moore}, and show that this implies that $\alpha$ is ergodic:

    We may fix some $y \in \M_+$ for which $\tau_{D^{-1}}(y) = 1$. First, consider some positive $z \in \M_+$ for which $\alpha_s(z) = z$ holds for all $s \in G$. Assuming \eqref{eqn:duflo-moore} to hold, and using $\alpha_s(z^{1/2}) = z^{1/2}, s \in G$, for all positive $x \in L^1(\M)_+$ we compute \[
        \tau(x) \tau_{D^{-1}}(z^{1/2} y z^{1/2})
        = \int_G \tau(x \alpha_s(z^{1/2} y z^{1/2})) \dif s
        = \int_G \tau(z^{1/2} x z^{1/2} \alpha_s(y)) \dif s
        = \tau_{z}(x) < \infty.
    \] The computation shows an equality of semifinite, normal weights \[
        \tau_z = c_z \tau, \quad \text{where } c_z = \tau_{D^{-1}}(z^{1/2} y z^{1/2}) \in \mathbb R_+.
    \] The uniqueness of Radon--Nikodym derivatives implies that $z = c_z 1$ is a scalar.

    Next, consider a more general $z \in \M$ with $\alpha_s(z) = z$ for all $s \in G$. If $z = z^*$ is self-adjoint, then by the spectral calculus also $\alpha_s(|z|) = |z|$ for all $s \in G$, so that both $z_\pm = \frac12(|z| \pm z) \in \M_+$ satisfy $\alpha_s(z_\pm) = z_\pm$ for all $s \in G$, hence $z_\pm \in \mathbb R_+ 1$. So $z = z_+ - z_- \in \mathbb R 1$. If $z$ is not necessarily self-adjoint, still $\operatorname{Re} z = \frac12(z + z^*)$ and $\operatorname{Im} z = \frac1{2i}(z - z^*)$ are self-adjoint and invariant under $\alpha_s$ for all $s \in G$, hence scalar elements $\mathbb R 1 \subseteq \M$, so $z = \operatorname{Re} z + i \operatorname{Im} z \in \mathbb C 1$ is a scalar.
\end{proof}

\printbibliography

\end{document}